\documentclass[12pt]{article}
\usepackage{amsmath,amsthm,amssymb}
\usepackage{amssymb,latexsym}

\newtheorem{theorem}{Theorem}

\newtheorem{lemma}{Lemma}

\begin{document}
\baselineskip=17pt

%
%

\title{\bf Pairs of square-free values of the type $\mathbf{n^2+1}$, $\mathbf{n^2+2}$}
\author{\bf S. I. Dimitrov}

\date{}

\maketitle

\begin{abstract}

In the present paper we show that there exist infinitely many
consecutive square-free numbers of the form $n^2+1$, $n^2+2$.
We also establish an asymptotic formula for the number of such square-free pairs
when $n$ does not exceed given sufficiently large positive number.
\quad\\
{\bf  2020 Math.\ Subject Classification}:  11L05 $\cdot$ 11N25 $\cdot$  11N37
\end{abstract}

\section{Notations}
\indent

Let $X$ be a sufficiently large positive number.
By $\varepsilon$ we denote an arbitrary small positive number, not necessarily the same in different occurrences.
As usual $\mu(n)$ is M\"{o}bius' function and $\tau(n)$ denotes the number of positive divisors of $n$.
Further $[t]$ and $\{t\}$ denote the integer part, respectively, the fractional part of $t$.
We shall use the convention that a congruence, $m\equiv n\,\pmod {d}$ will be written as $m\equiv n\,(d)$.
As usual $(m,n)$ is the greatest common divisor of $m$ and $n$.
The letter $p$  will always denote prime number.
We put
\begin{equation}\label{psit1}
\psi(t)=\{t\}-1/2\,.
\end{equation}
Moreover $e(t)$=exp($2\pi it$).
For $x, y \in\mathbb{R}$ we write $x\equiv y\,(1)$ when $x-y\in\mathbb{Z}$.
For any $n$ and $q$ such that $(n, q)=1$ we denote by $\overline{n}_q$
the inverse of $n$ modulo $q$.
The number of distinct prime factors of a natural number $n$ we denote by $\omega(n)$.
For any odd prime number $p$ we denote by $\left(\frac{\cdot}{p}\right)$  the Legendre symbol.
By $K(r,h)$ we shall denote the incomplete Kloosterman sum
\begin{equation}\label{Kloosterman}
K(r,h)=\sum\limits_{\alpha\leq x< \beta\atop{(x, r)=1}}e\left(\frac{h\overline{x}_{|r|}}{r}\right)\,,
\end{equation}
where
\begin{equation*}
h, r\in\mathbb{Z}, \quad hr\neq 0, \quad 0<\beta-\alpha\leq2|r|.
\end{equation*}

\section{Introduction and statement of the result}
\indent

In 1931 Estermann \cite{Estermann1} proved that there exist infinitely many
square-free numbers of the form $n^2+1$.
More precisely he proved that for $X \geq 2$  the asymptotic formula
\begin{equation*}
\sum\limits_{n\leq X}\mu^2(n^2+1)=c_0X+\mathcal{O}\left(X^{\frac{2}{3}}\log X\right)
\end{equation*}
holds.
Here
\begin{equation*}
c_0=\prod\limits_{p\equiv1\, (4 )}\left(1-\frac{2}{p^2}\right)\,.
\end{equation*}
Afterwards Heath-Brown \cite{Heath-Brown2} used a variant of the determinant method
and improved the remainder term in the formula of Estermann
with $\mathcal{O}\left(X^{7/12+\varepsilon}\right)$.

On the other hand  1932  Carlitz \cite{Carlitz} showed that  there exist infinitely many
pairs of consecutive square-free numbers.
More precisely he proved the asymptotic formula
\begin{equation}\label{Carlitz}
\sum\limits_{n\leq X}\mu^2(n)\mu^2(n+1)=\prod\limits_{p}\left(1-\frac{2}{p^2}\right)X
+\mathcal{O}\big(X^{\theta+\varepsilon}\big)\,,
\end{equation}
where $\theta=2/3$.
Formula \eqref{Carlitz} was sharpened by Heath-Brown \cite{Heath-Brown1}
to $\theta=7/11$ and by Reuss \cite{Reuss} to $\theta=(26+\sqrt{433})/81$.

The existence of  infinitely many consecutive square-free numbers
of a special form was demonstrated by the author in
\cite{Dimitrov1}, \cite{Dimitrov2},  \cite{Dimitrov3}, \cite{Dimitrov4}.
In particular in \cite{Dimitrov4} he proved that there exist infinitely many
consecutive square-free numbers of the form $x^2+y^2+1$, $x^2+y^2+2$.
While in \cite{Dimitrov4} the main role was played by the properties of Gauss sums,
in this paper we use a bijective correspondence between the number of
representations of numbers by binary quadratic form and the
the incongruent solutions of quadratic congruence.

Define
\begin{equation}\label{GammaX}
\Gamma(X)=\sum\limits_{1\leq n\leq X}\mu^2(n^2+1)\,\mu^2(n^2+2)\,,
\end{equation}
\begin{equation}\label{Sq1q2}
S(q_1, q_2)=\{ n\in\mathbb{N }\; : \; 1\leq n \leq q_1q_2, \;\;  n^2+1\equiv 0\,(q_1),\;\; n^2+2\equiv 0\,(q_2) \}
\end{equation}
and
\begin{equation}\label{lambdaq1q2}
\lambda(q_1, q_2)=\sum\limits_{n\in S(q_1, q_2)}1\,.
\end{equation}
We establish our result by combining the tasks of Estermann and Carlitz.
Thus we prove the following theorem.
\begin{theorem}\label{Theorem1}  For the sum $\Gamma(X)$ defined by \eqref{GammaX} the asymptotic formula
\begin{equation}\label{asymptoticformula}
\Gamma(X)=\sigma X+\mathcal{O}\left(X^{\frac{8}{9}+\varepsilon}\right)
\end{equation}
holds. Here
\begin{equation}\label{sigmaproduct}
\sigma=\prod\limits_{p>2}\left(1-\frac{\left(\frac{-1}{p}\right)+\left(\frac{-2}{p}\right)+2}{p^2}\right)\,.
\end{equation}
\end{theorem}
From Theorem \ref{Theorem1} it follows that there exist infinitely many
consecutive square-free numbers of the form $n^2+1$, $n^2+2$, where $n$ runs over naturals.

\section{Lemmas}
\indent

The first lemma we need gives us important expansions.
\begin{lemma}\label{expansion}
For any $M\geq2$, we have
\begin{equation*}
\psi(t)=-\sum\limits_{1\leq|m|\leq M}\frac{e(mt)}{2\pi i m}
+\mathcal{O}\big(f_M(t)\big)\,,
\end{equation*}
where $f_M(t)$ is a positive function of $t$ which is infinitely many
times differentiable and periodic with period 1.
It can be expanded into the Fourier series
\begin{equation*}
f_M(t)=\sum\limits_{m=-\infty}^{+\infty}b_{M}(m)e(m t)\,,
\end{equation*}
with coefficients $b_{M}(m)$ such that
\begin{equation*}
b_{M}(m)\ll\frac{\log M}{M}\quad \mbox{for all}\quad m
\end{equation*}
and
\begin{equation*}
\sum\limits_{|m|>M^{1+\varepsilon}}|b_{M}(m)|\ll M^{-A}\,.
\end{equation*}
Here $A > 0$ is arbitrarily large and the constant in the $\ll$ - symbol depends on $A$ and $\varepsilon$.
\end{lemma}
\begin{proof}
See (\cite{Tolev1}, Theorem 1).
\end{proof}
The next lemma we need is well-known.
\begin{lemma}\label{Wellknown}
Let $A, B\in\mathbb{Z}\setminus \{0\}$ and $(A, B)=1$. Then
\begin{equation*}
\frac{\overline{A}_{|B|}}{B}+\frac{\overline{B}_{|A|}}{A}\equiv \frac{1}{AB}\,\,(\,1\,).
\end{equation*}
\end{lemma}
\begin{proof}
See (\cite{Tolev2}, Lemma 17.5.1).
\end{proof}

\begin{lemma}\label{Weilsestimate}
For the sum denoted by \eqref{Kloosterman} the estimate
\begin{equation*}
K(r,h)\ll|r|^{\frac{1}{2}+\varepsilon}\,(r,h)^{\frac{1}{2}}
\end{equation*}
holds.
\end{lemma}
\begin{proof}
Follows easily from  A. Weil's estimate for the Kloosterman sum.
See (\cite{Iwaniec}, Ch. 11, Corollary 11.12).
\end{proof}

\begin{lemma}\label{Surjection}
Let $n\geq5$. There exists a bijective function from the solution set of the equation
\begin{equation}\label{Equation}
x^2+2y^2=n\,, \quad (x,y)=1\,, \quad x\in\mathbb{N}\,, \quad y\in\mathbb{Z}\setminus \{0\}
\end{equation}
to the incongruent solutions modulo $n$ of the congruence
\begin{equation}\label{Congruence}
z^2+2\equiv 0\,(n)\,.
\end{equation}
\end{lemma}
\begin{proof}
Let $F$ denote the set of ordered pairs $(x,y)$ satisfying \eqref{Equation}
and $E$ denote the set of solutions of the congruence \eqref{Congruence}.
We consider each residue class modulo $n$ with representatives satisfying \eqref{Congruence}
as one solution of \eqref{Congruence}.
Let $(x,y)\in F$. From \eqref{Equation}  it follows that $(n, y) = 1$.
Therefore there exists a unique residue class $z$ modulo $n$ such that
\begin{equation}\label{zyx}
zy\equiv  x\,(n)\,.
\end{equation}
For this class we have
\begin{equation*}
(z^2+2)y^2\equiv(zy)^2+2y^2\equiv x^2+2y^2\equiv0\,(n).
\end{equation*}
From the last congruence and $(n, y) = 1$ we deduce $z^2+2\equiv 0\,(n)$
which means that $z\in E$.
We define the map
\begin{equation}\label{map}
\beta : F\rightarrow E
\end{equation}
that associates to each pair  $(x,y)\in F$ the  residue class $z=x\overline{y}_n$ satisfying \eqref{zyx}. \\
We will first prove that the map \eqref{map} is a injection.
Let  $(x,y),\,(x',y')\in F$ that is
\begin{equation}\label{Systemxyx'y'}
\left|\begin{array}{cc}
x^2+2y^2=n\;\\
x'^2+2y'^2=n\\
\end{array}\right.\,,
\end{equation}
\begin{equation}\label{coprimexyx'y'}
(x,y)=(x',y')=1
\end{equation}
and
\begin{equation}\label{xyx'y'}
(x,y)\neq(x',y')\,.
\end{equation}
Assume that
\begin{equation}\label{betaxyx'y'}
\beta(x,y)=\beta(x',y')\,.
\end{equation}
Hence there exists $z\in E$ such that
\begin{equation}\label{System1}
\left|\begin{array}{cc}
zy\equiv  x\,(n)\quad\\
zy'\equiv  x'\,(n)\;\;\\
\end{array}\right..
\end{equation}
The system \eqref{System1} implies
\begin{equation}\label{xy'x'y}
xy'-x'y\equiv 0\,(n)\,.
\end{equation}
On the other hand \eqref{Systemxyx'y'} and $n\geq5$ yield                                                                                \begin{equation}\label{System2}
\left|\begin{array}{cc}
0<x, x'<\sqrt{n}\quad\;\\
0<|y|, |y'|<\sqrt{\frac{n}{2}}\\
\end{array}\right..
\end{equation}
We first consider the case $yy'>0$.
By  \eqref{System2} we derive
\begin{equation*}
\left|\begin{array}{cc}
0<xy'<\frac{n}{\sqrt{2}}\\
0<x'y<\frac{n}{\sqrt{2}}\\
\end{array}\right. \quad \quad \mbox{ or } \quad \quad
\left|\begin{array}{cc}
-\frac{n}{\sqrt{2}}<xy'<0\\
-\frac{n}{\sqrt{2}}<x'y<0\\
\end{array}\right.
\end{equation*}
and consequently
\begin{equation}\label{-nxy'x'yn}
-\frac{n}{\sqrt{2}}<xy'-x'y<\frac{n}{\sqrt{2}}\,.
\end{equation}
Now \eqref{xy'x'y} and \eqref{-nxy'x'yn} lead to
\begin{equation*}
xy'-x'y=0
\end{equation*}
which together with \eqref{coprimexyx'y'} gives us
\begin{equation}\label{xx'yy'}
x=x'\,, \quad y=y'\,.
\end{equation}
From \eqref{xyx'y'} and \eqref{xx'yy'} we get a contradiction.\\ Next we consider the case $yy'<0$.
By  \eqref{System2} we deduce
\begin{equation*}
\left|\begin{array}{cc}
0<xy'<\frac{n}{\sqrt{2}}\\
-\frac{n}{\sqrt{2}}<x'y<0\\
\end{array}\right.  \quad \quad \mbox{ or } \quad \quad
\left|\begin{array}{cc}
-\frac{n}{\sqrt{2}}<xy'<0\\
0<x'y<\frac{n}{\sqrt{2}}\\
\end{array}\right.
\end{equation*}
and therefore
\begin{equation}\label{-nxy'x'ynsqrt}
-n\sqrt{2}<xy'-x'y<n\sqrt{2}\,.
\end{equation}
Now \eqref{xy'x'y} and \eqref{-nxy'x'ynsqrt} lead to
\begin{equation}\label{xy'x'yn}
xy'-x'y=n  \quad \mbox{ or } \quad  xy'-x'y=-n \,.
\end{equation}
Raising to the second power one of the equations \eqref{xy'x'yn} we deduce
\begin{equation*}
x^2y'^2-2xx'yy'+x'^2y^2=n^2
\end{equation*}
which is equivalent to
\begin{equation*}
(x^2+2y^2)y'^2+(x'^2+2y'^2)y^2-4y^2y'^2-2xx'yy'=n^2\,.
\end{equation*}
The last equation, \eqref{Systemxyx'y'} and \eqref{coprimexyx'y'} assure us that \begin{equation}\label{2xx'2yy'n}
2(xx'+2yy')\equiv 0\,(n)\,.
\end{equation}
From \eqref{System2} we get
\begin{equation*}
\left|\begin{array}{cc}
0<xx'<n\\
-n<2y'y<0\\
\end{array}\right.
\end{equation*}
and therefore
\begin{equation}\label{-nxx'2yy'ynsqrt}
-n<xx'+2yy'<n\,.
\end{equation}
If $n$ is odd then \eqref{2xx'2yy'n} implies
\begin{equation}\label{xx'2yy'n}
xx'+2yy'\equiv 0\,(n)
\end{equation}
which together with \eqref{-nxx'2yy'ynsqrt} yield
\begin{equation}\label{xx'2yy'0}
xx'+2yy'= 0\,.
\end{equation}
On the other hand when $n$ is odd by \eqref{Systemxyx'y'} it follows that $x$ and $x'$ are odd
which contradicts \eqref{xx'2yy'0}.  Consequently $n$ cannot be odd. \\ Let $n$ be even.
Now \eqref{Systemxyx'y'} and \eqref{coprimexyx'y'} give us that  $x$, $x'$ are even,
$y$, $y'$ are odd and $4\nmid n$. These considerations and \eqref{2xx'2yy'n} lead to  \eqref{xx'2yy'n}
which together with \eqref{-nxx'2yy'ynsqrt} imply  \eqref{xx'2yy'0}.
But equation \eqref{xx'2yy'0} for even $x$ and $x'$ means that $yy'$ is even which contradicts \eqref{coprimexyx'y'}.
Consequently $n$ cannot be even.
The resulting contradictions show that the assumption \eqref{betaxyx'y'} is not true.
This proves the injectivity of $\beta$.

It remains to show that the map \eqref{map} is a surjection. Let $z\in E$.
From Dirichlet's approximation theorem it follows that there exist integers $a$ and $q$ such that
\begin{equation}\label{Dirichlet}
\left|\frac{z}{n}-\frac{a}{q}\right|<\frac{1}{q\sqrt{n}}\,,
  \quad\quad 1\leq q\leq \sqrt{n},\quad\quad (a,\,q)=1\,.
\end{equation}
Replace
\begin{equation}\label{r}
r=zq-an\,.
\end{equation}
Hence
\begin{equation}\label{r^2+2q^2}
r^2+2q^2=z^2q^2-2zqan+a^2n^2+2q^2\equiv(z^2+2)q^2\,(n)\,.
\end{equation}
From \eqref{Congruence} and \eqref{r^2+2q^2} it follows
\begin{equation}\label{r^2+2q^2modn}
r^2+2q^2\equiv0\,(n)\,.
\end{equation}
By \eqref{Dirichlet} and \eqref{r} we deduce
\begin{equation}\label{rest}
|r|<\sqrt{n}\,.
\end{equation}
Using \eqref{Dirichlet} and \eqref{rest} we obtain
\begin{equation}\label{r^2+2q^2est}
0<r^2+2q^2<3n\,.
\end{equation}
Bearing in mind \eqref{r^2+2q^2modn} and \eqref{r^2+2q^2est} we conclude that
$r^2+2q^2=n$ or $r^2+2q^2=2n$.

Consider two cases.

\textbf{Case 1}
\begin{equation}\label{Case1}
r^2+2q^2=n\,.
\end{equation}
From \eqref{r} and \eqref{Case1} we get
\begin{equation*}
n=(zq-an)^2+2q^2=(zq-an)zq-(zq-an)an+2q^2=(zq-an)zq-ran+2q^2
\end{equation*}
and therefore
\begin{equation}\label{ra+1}
ra+1=kq,
\end{equation}
where
\begin{equation}\label{kznqaz}
k=\frac{z^2+2}{n}q-az.
\end{equation}
By \eqref{Congruence} and \eqref{kznqaz} it follows that  $k\in\mathbb{Z}$ and taking into account \eqref{ra+1} we deduce
\begin{equation}\label{rq1}
(r,q)=1.
\end{equation}
Using \eqref{Case1}, \eqref{rq1} and $n\geq5$ we establish that $r\neq0$. \\
Consider first $r>0$. Replace
\begin{equation}\label{xy1}
x=r, \quad y=q.
\end{equation}
From \eqref{Case1}, \eqref{rq1} and \eqref{xy1} it follows that $(x,y)\in F$.
Also \eqref{r} and  \eqref{xy1} give us \eqref{zyx}. Consequently $\beta(x,y)=z$.  \\
Next we consider $r<0$. Put
\begin{equation}\label{xy2}
x=-r, \quad y=-q.
\end{equation}
Again \eqref{Case1}, \eqref{rq1} and \eqref{xy2} lead to $(x,y)\in F$.
As well from \eqref{r} and  \eqref{xy2} follows \eqref{zyx}. Therefore $\beta(x,y)=z$.

\textbf{Case 2}
\begin{equation}\label{Case2}
r^2+2q^2=2n\,.
\end{equation}
From \eqref{r} and \eqref{Case2} we find
\begin{equation*}
2n=(zq-an)^2+2q^2=(zq-an)zq-(zq-an)an+2q^2=(zq-an)zq-ran+2q^2
\end{equation*}
and thus
\begin{equation}\label{ra+2}
ra+2=kq\,,
\end{equation}
where $k$ is denoted by \eqref{kznqaz}.
From \eqref{ra+2} we conclude
\begin{equation}\label{rq2}
(r,q)\leq2.
\end{equation}
By \eqref{Case2}, \eqref{rq2} and $n\geq5$ we deduce that $r\neq0$.
On the other hand from \eqref{Case2} it follows that $r$ is even.
We  replace $r=2r_0$ in \eqref{Case2} and obtain
\begin{equation}\label{nqr0}
q^2+2r_0^2=n\,.
\end{equation}
We shall verify that
\begin{equation}\label{r0q1}
(r_0,q)=1\,.
\end{equation}
If we assume that $(r_0,q)>1$ then \eqref{rq2} gives us
\begin{equation}\label{r0q2}
(r_0,q)=2\,.
\end{equation}
From \eqref{nqr0} and \eqref{r0q2} it follows
\begin{equation}\label{4|n}
n\equiv0\,(4)\,.
\end{equation}
Finally \eqref{Congruence} and \eqref{4|n}  imply
\begin{equation*}
z^2+2\equiv 0\,(4)
\end{equation*}
which is impossible. This proves \eqref{r0q1}.
No matter whether $r$ is positive or negative we replace
\begin{equation}\label{xy3}
x=q, \quad y=-r_0\,.
\end{equation}
Using \eqref{nqr0}, \eqref{r0q1} and \eqref{xy3} we deduce that $(x,y)\in F$.
By \eqref{r} and \eqref{xy3} we get
\begin{equation}\label{2zyx1}
2(zy-x)=-2(zr_0+q)=-zr-2q=-(z^2+2)q+zan
\end{equation}
From \eqref{Congruence} and \eqref{2zyx1} we conclude
\begin{equation}\label{2zyx2}
2(zy-x)\equiv 0\,(n)\,.
\end{equation}
If $n$ is odd then \eqref{2zyx2} gives us \eqref{zyx}. Consequently $\beta(x,y)=z$. \\
Let $n$ be even. Since \eqref{4|n} is impossible then
\begin{equation}\label{nn0}
n=2n_0, \quad  n_0 \hbox{ is odd}.
\end{equation}
By \eqref{nqr0} and \eqref{nn0} it follows
\begin{equation}\label{qeven}
q\equiv0\,(2)\,,
\end{equation}
i.e $q$ is even.
On the other hand \eqref{Congruence} and \eqref{nn0} imply that
\begin{equation}\label{zeven}
z\equiv0\,(2)\,,
\end{equation}
i.e $z$ is even.
Now \eqref{xy3}, \eqref{qeven} and  \eqref{zeven} give us
\begin{equation}\label{zyxeven}
zy-x\equiv0\,(2)\,,
\end{equation}
i.e $zy-x$ is even.
Finally from \eqref{2zyx2}, \eqref{nn0} and \eqref{zyxeven} we obtain \eqref{zyx}.
Therefore $\beta(x,y)=z$.

The lemma is proved.
\end{proof}

\section{Proof of the theorem}
\indent

Using \eqref{GammaX} and the well-known identity $\mu^2(n)=\sum_{d^2|n}\mu(d)$ we get
\begin{equation}\label{GammaXdecomp}
\Gamma(X)=\sum\limits_{d_1, d_2\atop{(d_1, d_2)=1}}\mu(d_1)\mu(d_2)
\sum\limits_{1\leq n\leq X\atop{n^2+1\equiv 0\,(d_1^2)\atop{n^2+2\equiv 0\,(d_2^2)}}}1
=\Gamma_1(X)+\Gamma_2(X)\,,
\end{equation}
where
\begin{align}
\label{Gamma1}
&\Gamma_1(X)=\sum\limits_{d_1 d_2\leq z\atop{(d_1, d_2)=1}}\mu(d_1)\mu(d_2)\Sigma(X, d_1^2, d_2^2)\,,\\
\label{Gamma2}
&\Gamma_2(X)=\sum\limits_{d_1 d_2>z\atop{(d_1, d_2)=1}}\mu(d_1)\mu(d_2)\Sigma(X, d_1^2, d_2^2)\,,\\
\label{Sigma}
&\Sigma(X, d_1^2, d_2^2)=\sum\limits_{1\leq n\leq X\atop{n^2+1\equiv 0\,(d_1^2)\atop{n^2+2\equiv 0\,(d_2^2)}}}1\,,\\
\label{z}
&\sqrt{X}\leq z< X\,,
\end{align}
where $z$ is to be chosen later.

\subsection{Estimation of $\mathbf{\Gamma_1(X)}$}
\indent

Suppose that $q_1=d_1^2$, $q_2=d_2^2$, where $d_1$ and $d_2$ are square-free, $(q_1, q_2)=1$ and $d_1 d_2\leq z$.
Denote
\begin{equation}\label{Omega}
\Omega(X, q_1, q_2, n)=\sum\limits_{m\leq X\atop{m\equiv n\,(q_1 q_2)}}1\,.
\end{equation}
Using \eqref{Sq1q2}, \eqref{Sigma} and \eqref{Omega} we obtain upon partitioning
the sum \eqref{Sigma} into residue classes modulo $q_1q_2$
\begin{equation}\label{SigmaOmega}
\Sigma(X, q_1, q_2)=\sum\limits_{n\in S(q_1, q_2)}
\Omega(X, q_1, q_2, n)\,.
\end{equation}
It is easy to see that
\begin{equation}\label{Omegaest}
\Omega(X, q_1, q_2, n)=\frac{X}{q_1 q_2}+\mathcal{O}(1)\,.
\end{equation}
From  \eqref{lambdaq1q2}, \eqref{SigmaOmega} and \eqref{Omegaest} we find
\begin{equation}\label{Sigmaest1}
\Sigma(X, q_1, q_2)
=X\frac{\lambda(q_1, q_2)}{q_1 q_2}+\mathcal{O}\big(\lambda(q_1, q_2)\big)\,.
\end{equation}
Taking into account \eqref{Sq1q2}, \eqref{lambdaq1q2}, Chinese remainder theorem
and that the number of solutions of the congruence $n^2\equiv a\,(q_1q_2)$
is less than or equal to $\tau(q_1q_2)$ we get
\begin{equation}\label{lambdaq1q2est}
\lambda(q_1, q_2)\ll\tau(q_1q_2).
\end{equation}
From \eqref{Sigmaest1}, \eqref{lambdaq1q2est} and the inequalities
\begin{equation*}
\tau(q_1q_2)\ll (q_1q_2)^\varepsilon\ll X^\varepsilon
\end{equation*}
it follows
\begin{equation}\label{Sigmaest2}
\Sigma(X, q_1, q_2)
=X\frac{\lambda(q_1, q_2)}{q_1 q_2}+\mathcal{O}\big(X^\varepsilon\big)\,.
\end{equation}
Bearing in mind  \eqref{Gamma1}, \eqref{z} and \eqref{Sigmaest2} we obtain
\begin{align}\label{GammaX1est1}
\Gamma_1(X)&=X\sum\limits_{d_1 d_2\leq z\atop{(d_1, d_2)=1}}
 \frac{\mu(d_1)\mu(d_2) \lambda(d^2_1, d^2_2)}{d^2_1 d^2_2}
+\mathcal{O}\big(zX^\varepsilon \big)\nonumber\\
&=\sigma X-X \sum\limits_{d_1 d_2>z\atop{(d_1, d_2)=1}}
 \frac{\mu(d_1)\mu(d_2) \lambda(d^2_1, d^2_2)}{d^2_1 d^2_2}
+\mathcal{O}\big(zX^\varepsilon\big)\,,
\end{align}
where
\begin{equation}\label{sigmasum}
\sigma=\sum\limits_{d_1, d_2=1\atop{(d_1, d_2)=1}}^\infty
 \frac{\mu(d_1)\mu(d_2)\lambda(d^2_1, d^2_2)}{d^2_1 d^2_2}\,.
\end{equation}
Using \eqref{lambdaq1q2est} we find
\begin{equation}\label{d1d2>est}
\sum\limits_{d_1 d_2>z\atop{(d_1, d_2)=1}}\frac{\mu(d_1)\mu(d_2)\lambda(d^2_1, d^2_2)}{d^2_1 d^2_2}
\ll\sum\limits_{d_1 d_2>z\atop{(d_1, d_2)=1}}\frac{(d_1 d_2)^{\varepsilon}}{(d_1 d_2)^2}
\ll\sum\limits_{n>z}\frac{\tau(n)}{n^{2-\varepsilon}}\ll z^{\varepsilon-1}\,.
\end{equation}
It remains to see that the product  \eqref{sigmaproduct} and the sum \eqref{sigmasum} coincide.
From the definition \eqref{lambdaq1q2} it follows that the function $\lambda(q_1, q_2)$ is multiplicative, i.e if
\begin{equation*}
(q_1 q_2, q_3 q_4)=(q_1, q_2)=(q_3, q_4)=1
\end{equation*}
then
\begin{equation}\label{multiplicative}
\lambda(q_1 q_2, q_3 q_4)=\lambda(q_1, q_3)\lambda(q_2, q_4).
\end{equation}
The proof is elementary and we leave it to the reader.

From the property \eqref{multiplicative} and  $(d_1, d_2)=1$  it follows
\begin{equation}\label{lambdad1d2}
\lambda(d^2_1, d^2_2)=\lambda(d^2_1, 1) \lambda(1, d^2_2)\,.
\end{equation}
Bearing in mind \eqref{sigmasum} and \eqref{lambdad1d2} we get
\begin{equation}\label{sigmasumest1}
\sigma=\sum\limits_{d_1=1}^\infty\frac{\mu(d_1)\lambda(d^2_1, 1)}{d_1^2}
\sum\limits_{d_2=1}^\infty\frac{\mu(d_2)\lambda(1, d^2_2)}{d_2^2}f_{d_1}(d_2)\,,
\end{equation}
where
\begin{equation*}
f_{d_1}(d_2)=\begin{cases}1 \;\; \text{ if }\; (d_1, d_2)=1\,,\\
0 \;\; \mbox{ if } \; (d_1, d_2)>1\,.
\end{cases}
\end{equation*}
Clearly the function
\begin{equation*}
\frac{\mu(d_2)\lambda(1, d^2_2)}{d_2^2}f_{d_1}(d_2)
\end{equation*}
is multiplicative with respect to $d_2$ and the series
\begin{equation*}
\sum\limits_{d_2=1}^\infty\frac{\mu(d_2)\lambda(1, d^2_2)}{d_2^2}f_{d_1}(d_2)
\end{equation*}
is absolutely convergent.

Applying the Euler product we obtain
\begin{align}\label{Eulerproduct}
\sum\limits_{d_2=1}^\infty\frac{\mu(d_2)\lambda(1, d^2_2)}{d_2^2}f_{d_1}(d_2)&=
\prod\limits_{p\nmid d_1}\left(1-\frac{\lambda(1, p^2)}{p^2}\right)\nonumber\\
&=\prod\limits_{p}\left(1-\frac{\lambda(1, p^2)}{p^2}\right)
\prod\limits_{p|d_1}\left(1-\frac{\lambda(1, p^2)}{p^2}\right)^{-1}\,.
\end{align}
From \eqref{sigmasumest1} and \eqref{Eulerproduct} it follows
\begin{align}\label{sigmasumest2}
\sigma&=\sum\limits_{d_1=1}^\infty\frac{\mu(d_1)\lambda(d^2_1, 1)}{d_1^2}
\prod\limits_{p}\left(1-\frac{\lambda(1, p^2)}{p^2}\right)
\prod\limits_{p|d_1}\left(1-\frac{\lambda(1, p^2)}{p^2}\right)^{-1}\nonumber\\
&=\prod\limits_{p}\left(1-\frac{\lambda(1, p^2)}{p^2}\right)
\sum\limits_{d_1=1}^\infty\frac{\mu(d_1)\lambda(d^2_1, 1)}{d_1^2}
\prod\limits_{p|d_1}\left(1-\frac{\lambda(1, p^2)}{p^2}\right)^{-1}\,.
\end{align}
Obviously  the function
\begin{equation*}
\frac{\mu(d_1)\lambda(d^2_1, 1)}{d_1^2}
\prod\limits_{p|d_1}\left(1-\frac{\lambda(1, p^2)}{p^2}\right)^{-1}
\end{equation*}
is multiplicative with respect to $d_1$ and the series
\begin{equation*}
\sum\limits_{d_1=1}^\infty\frac{\mu(d_1)\lambda(d^2_1, 1)}{d_1^2}
\prod\limits_{p|d_1}\left(1-\frac{\lambda(1, p^2)}{p^2}\right)^{-1}
\end{equation*}
is absolutely convergent.

Applying again the Euler product from \eqref{lambdaq1q2} and \eqref{sigmasumest2} we find
\begin{align}\label{sigmasumest3}
\sigma&=
\prod\limits_{p}\left(1-\frac{\lambda(1, p^2)}{p^2}\right)
\prod\limits_{p}\left(1-\frac{\lambda(p^2, 1)}{p^2}\left(1-\frac{\lambda(1, p^2)}{p^2}\right)^{-1}\right)\nonumber\\
&=\prod\limits_{p}\left(1-\frac{\lambda(p^2, 1)+\lambda(1, p^2)}{p^2}\right)
=\prod\limits_{p>2}\left(1-\frac{\left(\frac{-1}{p}\right)+\left(\frac{-2}{p}\right)+2}{p^2}\right)\,.
\end{align}
Bearing in mind \eqref{z}, \eqref{GammaX1est1}, \eqref{d1d2>est} and \eqref{sigmasumest3} we get
\begin{equation}\label{GammaX1est2}
\Gamma_1(X)=\sigma X+\mathcal{O}\big(zX^\varepsilon\big)\,,
\end{equation}
where $\sigma$ is given by the product \eqref{sigmaproduct}.

\subsection{Estimation of $\mathbf{\Gamma_2(X)}$}
\indent

Using \eqref{Gamma2}, \eqref{Sigma} and splitting the range of $d_1$ and $d_2$
into dyadic subintervals of the form $D_1\leq d_1<2D_1$, $D_2\leq d_2<2D_2$
we write
\begin{equation}\label{Gamma2est1}
\Gamma_2(X)\ll(\log X)^2\sum\limits_{n\leq X}
\sum\limits_{D_1\leq d_1<2D_1\atop{n^2+1\equiv 0\,(d_1^2)}}
\sum\limits_{D_2\leq d_2<2D_2\atop{n^2+2\equiv 0\,(d_2^2)}}1\,,
\end{equation}
where
\begin{equation}\label{DT}
\frac{1}{2}\leq D_1,D_2\leq\sqrt{X^2+2}\,,\quad D_1D_2>\frac{z}{4}\,.
\end{equation}
On the one hand \eqref{Gamma2est1} gives us
\begin{equation}\label{Gamma2est2}
\Gamma_2(X)\ll X^{\varepsilon}\Sigma_1\,,
\end{equation}
where
\begin{equation}\label{Sigma1}
\Sigma_1=\sum\limits_{n\leq X}\sum\limits_{D_1\leq d_1<2D_1\atop{n^2+1\equiv 0\,(d_1^2)}}1\,.
\end{equation}
On the other hand \eqref{Gamma2est1} implies
\begin{equation}\label{Gamma2est3}
\Gamma_2(X)\ll X^{\varepsilon}\Sigma_2\,,
\end{equation}
where
\begin{equation}\label{Sigma2}
\Sigma_2=\sum\limits_{n\leq X}\sum\limits_{D_2\leq d_2<2D_2\atop{n^2+2\equiv 0\,(d_2^2)}}1\,.
\end{equation}

\textbf{Estimation of $\mathbf{\Sigma_1}$}

Define
\begin{align}
\label{N1set}
&\mathcal{N}_1(d)=\{n\in\mathbb{N }\; : \; 1\leq n \leq d, \;\;  n^2+1\equiv 0\,(d) \}\,,\\
\label{N'1set}
&\mathcal{N}'_1(d)=\{n\in\mathbb{N }\; : \; 1\leq n \leq d^2, \;\;  n^2+1\equiv 0\,(d^2) \}\,.
\end{align}
By \eqref{Sigma1} and \eqref{N'1set} we obtain
\begin{align}\label{Sigma1est1}
\Sigma_1&=\sum\limits_{D_1\leq d_1<2D_1}\sum\limits_{n\in \mathcal{N}'_1(d_1)}
\sum\limits_{m\leq X\atop{m\equiv n\,(d^2_1)}}1
=\sum\limits_{D_1\leq d_1<2D_1}\sum\limits_{n\in \mathcal{N}'_1(d_1)}
\Bigg(\left[\frac{X-n}{d^2_1}\right]-\left[\frac{-n}{d^2_1}\right]\Bigg) \nonumber\\
&=\sum\limits_{D_1\leq d_1<2D_1}\sum\limits_{n\in \mathcal{N}'_1(d_1)}
\Bigg(\frac{X}{d^2_1}+\psi\left(\frac{-n}{d^2_1}\right)-\psi\left(\frac{X-n}{d^2_1}\right)\Bigg)\nonumber\\
&\ll X^{1+\varepsilon}D_1^{-1}+|\Sigma'_1|+|\Sigma''_1|\,,
\end{align}
where
\begin{align}
\label{Sigma'1}
&\Sigma'_1=\sum\limits_{D_1\leq d_1<2D_1}\sum\limits_{n\in \mathcal{N}'_1(d_1)}\psi\left(\frac{-n}{d^2_1}\right)\,,\\
\label{Sigma''1}
&\Sigma''_1=\sum\limits_{D_1\leq d_1<2D_1}\sum\limits_{n\in \mathcal{N}'_1(d_1)}\psi\left(\frac{X-n}{d^2_1}\right)
\end{align}
and $\psi(t)$ is defined by \eqref{psit1}.

Firstly we consider the sum $\Sigma'_1$.
We note that the sum over $n$ in  \eqref{Sigma'1} does not contain terms with $n=\frac{d_1^2}{2}$ and $n=d_1^2$ .
Moreover for any $n$ satisfying the congruences $n^2+1\equiv 0\,(d_1^2)$  and such that
$1\leq n<\frac{d_1^2}{2}$ the number $d_1^2-n$ satisfies the same congruence
and we have $\psi\left(\frac{-n}{d_1^2}\right)+\psi\left(\frac{-(d_1^2-n)}{d_1^2}\right)=0$.
Bearing in mind these arguments for the sum $\Sigma'_1$ denoted by \eqref{Sigma'1}
we have that
\begin{equation}\label{Sigma'1est}
\Sigma'_1=0\,.
\end{equation}
Next we consider the sum $\Sigma''_1$ denoted by \eqref{Sigma''1}. Let $D_1\leq X^{\frac{1}{2}}$. The trivial estimation gives us
\begin{equation}\label{Sigma''1est1}
\Sigma''_1\ll\sum\limits_{D_1\leq d_1<2D_1}d^\varepsilon_1\ll  X^{\frac{1}{2}+\varepsilon}.
\end{equation}
Let
\begin{equation}\label{D1>}
D_1> X^{\frac{1}{2}}.
\end{equation}
From the theory of the quadratic congruences we know that when $\#\mathcal{N}'_1(d)\neq0$ then $d$ is odd and
\begin{equation}\label{might}
\#\mathcal{N}_1(d)=\#\mathcal{N}'_1(d)=2^{\omega(d)}\,.
\end{equation}
Denote
\begin{equation}\label{k}
k=2^{\omega(d)}\,,
\end{equation}
\begin{equation}\label{solutions}
n_1, \ldots, n_k\in\mathcal{N}_1(d_1)\,,\quad n'_1, \ldots, n'_k\in\mathcal{N}'_1(d_1)\,.
\end{equation}
From  \eqref{N1set}, \eqref{N'1set}, \eqref{D1>} -- \eqref{solutions} and $d\geq D_1> X^{\frac{1}{2}}$  it follows
\begin{align}\label{Prehod}
&\sum\limits_{n\in \mathcal{N}'_1(d_1)}\psi\left(\frac{X-n}{d^2_1}\right)
=\sum\limits_{n\in \mathcal{N}'_1(d_1)}\left(\frac{X-n}{d^2_1}-\frac{1}{2}\right)\nonumber\\
&=\sum\limits_{n\in \mathcal{N}'_1(d_1)}\left(\frac{X}{d^2_1}-\frac{1}{2}\right)
-\frac{n'_1+\cdots+n'_{k/2}+(d_1^2-n'_1)+\cdots+(d_1^2-n'_{k/2})}{d^2_1}\nonumber\\
&=\sum\limits_{n\in \mathcal{N}_1(d_1)}\left(\frac{X}{d^2_1}-\frac{1}{2}\right)-
\frac{n_1+\cdots+n_{k/2}+(d_1-n_1)+\cdots+(d_1-n_{k/2})}{d_1}\nonumber\\
&=\sum\limits_{n\in \mathcal{N}_1(d_1)}\left(\frac{X}{d^2_1}-\frac{1}{2}\right)
-\sum\limits_{n\in \mathcal{N}_1(d_1)}\frac{n}{d_1}\nonumber\\
&=\sum\limits_{n\in \mathcal{N}_1(d_1)}\left(\frac{X}{d^2_1}-\frac{\sqrt{X}}{d_1}\right)
+\sum\limits_{n\in \mathcal{N}_1(d_1)}\left(\frac{\sqrt{X}-n}{d_1}-\frac{1}{2}\right)\nonumber\\
&=\sum\limits_{n\in \mathcal{N}_1(d_1)}\left(\frac{X}{d^2_1}-\frac{\sqrt{X}}{d_1}\right)
+\sum\limits_{n\in \mathcal{N}_1(d_1)}\psi\left(\frac{\sqrt{X}-n}{d_1}\right).
\end{align}
By \eqref{Sigma''1}, \eqref{D1>}  and \eqref{Prehod} we obtain
\begin{equation}\label{Sigma''1est2}
\Sigma''_1\ll  X^{\frac{1}{2}+\varepsilon}+|\Sigma_3|\,,
\end{equation}
where
\begin{equation}\label{Sigma3}
\Sigma_3=\sum\limits_{D_1\leq d_1<2D_1}\sum\limits_{n\in \mathcal{N}_1(d_1)}\psi\left(\frac{\sqrt{X}-n}{d_1}\right)\,.
\end{equation}
Using  \eqref{Sigma3} and Lemma \ref{expansion} with
\begin{equation}\label{M1}
M_1=X^{\frac{1}{2}}
\end{equation}
we find
\begin{equation*}
\Sigma_3=\sum\limits_{D_1\leq d_1<2D_1}\sum\limits_{n\in \mathcal{N}_1(d_1)}
\Bigg(-\sum\limits_{1\leq|m|\leq M_1}\frac{e\left(m\frac{\sqrt{X}-n}{d_1}\right)}{2\pi i m}
+\mathcal{O}\left(f_{M_1}\left(\frac{\sqrt{X}-n}{d_1}\right)\right)\Bigg).
\end{equation*}
Arguing as in (Tolev \cite{Tolev2}, Theorem 17.1.1) we deduce
\begin{equation}\label{Sigma3est}
\Sigma_3\ll X^\varepsilon\Big(D_1M_1^{-1}+D_1^{\frac{3}{4}}+X^{\frac{1}{2}}M_1D_1^{-\frac{1}{4}}\Big).
\end{equation}
Bearing in mind \eqref{Sigma1est1}, \eqref{Sigma'1est}, \eqref{Sigma''1est1}, \eqref{Sigma''1est2},
\eqref{M1} and \eqref{Sigma3est} we get
\begin{equation}\label{Sigma1est2}
\Sigma_1\ll X^{1+\varepsilon}D_1^{-\frac{1}{4}}.
\end{equation}

\textbf{Estimation of $\mathbf{\Sigma_2}$}

Our argument is a modification of (Tolev \cite{Tolev2}, Theorem 17.1.1) argument.

Define
\begin{equation}\label{N2set}
\mathcal{N}_2(d)=\{n\in\mathbb{N }\; : \; 1\leq n \leq d, \;\;  n^2+2\equiv 0\,(d) \}.
\end{equation}
Working as in $\Sigma_1$ from \eqref{Sigma2} and \eqref{N2set} we find
\begin{equation}\label{Sigma2est1}
\Sigma_2\ll X^{1+\varepsilon}D_2^{-1}
\end{equation}
for $D_2\leq X^{\frac{1}{2}}$ and
\begin{equation}\label{Sigma2est2}
\Sigma_2\ll X^{\frac{1}{2}+\varepsilon}+|\Sigma_4|
\end{equation}
for
\begin{equation}\label{D2>}
D_2> X^{\frac{1}{2}},
\end{equation}
where
\begin{equation}\label{Sigma4}
\Sigma_4=\sum\limits_{D_2\leq d_2<2D_2}\sum\limits_{n\in \mathcal{N}_2(d_2)}\psi\left(\frac{\sqrt{X}-n}{d_2}\right)\,.
\end{equation}
From  \eqref{Sigma4} and Lemma \ref{expansion} with
\begin{equation}\label{M2}
M_2=X^{\frac{1}{2}}
\end{equation}
we obtain
\begin{align}\label{Sigma4est1}
\Sigma_4&=\sum\limits_{D_2\leq d_2<2D_2}\sum\limits_{n\in \mathcal{N}_2(d_2)}
\Bigg(-\sum\limits_{1\leq|m|\leq M_2}\frac{e\left(m\left(\frac{\sqrt{X}-n}{d_2}\right)\right)}{2\pi i m}
+\mathcal{O}\left(f_{M_2}\left(\frac{\sqrt{X}-n}{d_2}\right)\right)\Bigg)\nonumber\\
&=\Sigma_5+\Sigma_6,
\end{align}
where
\begin{align}
\label{Sigma5}
&\Sigma_5=\sum\limits_{1\leq|m|\leq M_2}\frac{\Theta_m}{2\pi i m}\,,\\
\label{Thetam}
&\Theta_m=\sum\limits_{D_2\leq d_2<2D_2}e\left(\frac{\sqrt{X}m}{d_2}\right)
\sum\limits_{n\in \mathcal{N}_2(d_2)}e\left(-\frac{nm}{d_2}\right)\,,\\
\label{Sigma6}
&\Sigma_6=\sum\limits_{D_2\leq d_2<2D_2}\sum\limits_{n\in \mathcal{N}_2(d_2)}f_{M_2}\left(\frac{\sqrt{X}-n}{d_2}\right)\,.
\end{align}
By  \eqref{Thetam}, \eqref{Sigma6} and Lemma \ref{expansion} it follows
\begin{align}\label{Sigma6est}
\Sigma_6&=\sum\limits_{D_2\leq d_2<2D_2}\sum\limits_{n\in \mathcal{N}_2(d_2)}
\sum\limits_{m=-\infty}^{+\infty}b_{M_2}(m)e\left(\frac{\sqrt{X}-n}{d_2}m\right)
=\sum\limits_{m=-\infty}^{+\infty}b_{M_2}(m)\Theta_m\nonumber\\
&\ll\frac{\log M_2}{M_2}|\Theta_0|+\frac{\log M_2}{M_2}\sum\limits_{1\leq|m|\leq M^{1+\varepsilon}_2 }|\Theta_m|
+\sum\limits_{|m|> M^{1+\varepsilon}_2}|b_{M_2}(m)||\Theta_m|\nonumber\\
&\ll\frac{\log M_2}{M_2}D_2^{1+\varepsilon}+\frac{\log M_2}{M_2}\sum\limits_{1\leq m\leq M^{1+\varepsilon}_2 }|\Theta_m|
+D_2^{1+\varepsilon}\sum\limits_{|m|> M^{1+\varepsilon}_2}|b_{M_2}(m)|\nonumber\\
&\ll\frac{\log M_2}{M_2}D_2^{1+\varepsilon}+\frac{\log M_2}{M_2}\sum\limits_{1\leq m\leq M^{1+\varepsilon}_2 }|\Theta_m|\,.
\end{align}
Using \eqref{Sigma4est1}, \eqref{Sigma5} and \eqref{Sigma6est} we get
\begin{equation}\label{Sigma4est2}
\Sigma_4\ll X^{\varepsilon}\left(\frac{D_2}{M_2}+\sum\limits_{1\leq m\leq M^{1+\varepsilon}_2 }\frac{|\Theta_m|}{m}\right).
\end{equation}
Define
\begin{equation}\label{Fdset}
\mathcal{F}(d)=\{(u,v)\; : \; u^2+2v^2=d, \;\;  (u,v)=1, \;\; u\in\mathbb{N}, \;\;  v\in\mathbb{Z}\setminus \{0\} \}.
\end{equation}
According to Lemma \ref{Surjection}  there exists a bijection
\begin{equation*}
\beta : \mathcal{F}(d)\rightarrow \mathcal{N}_2(d)
\end{equation*}
from $\mathcal{F}(d)$ to  $\mathcal{N}_2(d)$ defined by \eqref{N2set}
that associates to each couple $(u,v)\in \mathcal{F}(d)$ the element $n\in\mathcal{N}_2(d)$  satisfying
\begin{equation}\label{nvu}
nv\equiv  u\,(d)\,.
\end{equation}
Now  \eqref{nvu}  gives us
\begin{equation*}
n_{u,v}\equiv  u\overline{v}_{d}\,(d)
\end{equation*}
and therefore
\begin{equation}\label{nuvd}
\frac{n_{u,v}}{d}\equiv  u\frac{\overline{v}_{u^2+2v^2}}{u^2+2v^2}\,\,(1)\,.
\end{equation}
Bearing in mind \eqref{nuvd} and  Lemma \ref{Wellknown}  we deduce
\begin{align}
\label{nuvd1}
&\frac{n_{u,v}}{d}\equiv \frac{u}{v(u^2+2v^2)}-\frac{\overline{u}_{|v|}}{v}\,\,(1)\,,\\
\label{nuvd2}
&\frac{n_{u,v}}{d}\equiv -\frac{2v}{u(u^2+2v^2)}+\frac{\overline{v}_{u}}{u}\,\,(1)\,.
\end{align}
From \eqref{Thetam}, \eqref{Fdset}, \eqref{nuvd1} and \eqref{nuvd2}  we find
\begin{align}\label{Thetamest1}
\Theta_m&=\sum\limits_{D_2\leq d_2<2D_2}e\left(\frac{m\sqrt{X}}{d_2}\right)
\sum\limits_{(u,v)\in\mathcal{F}(d_2)}e\left(-\frac{n_{u,v}}{d_2}m\right)\nonumber\\
&=\sum\limits_{D_2\leq d_2<2D_2}e\left(\frac{m\sqrt{X}}{d_2}\right)
\sum\limits_{(u,v)\in\mathcal{F}(d_2)\atop{0<u<|v|}}e\left(-\frac{mu}{v(u^2+2v^2)}+\frac{m\overline{u}_{|v|}}{v}\right)\nonumber\\
&+\sum\limits_{D_2\leq d_2<2D_2}e\left(\frac{m\sqrt{X}}{d_2}\right)
\sum\limits_{(u,v)\in\mathcal{F}(d_2)\atop{0<|v|<u}}e\left(\frac{2mv}{u(u^2+2v^2)}-\frac{m\overline{v}_{u}}{u}\right)\nonumber\\
&=\sum\limits_{D_2\leq u^2+2v^2<2D_2\atop{0<u<|v|\atop{(u, v)=1}}}
e\left(\frac{m\sqrt{X}}{u^2+2v^2}-\frac{mu}{v(u^2+2v^2)}+\frac{m\overline{u}_{|v|}}{v}\right)\nonumber\\
&+\sum\limits_{D_2\leq u^2+2v^2<2D_2\atop{0<|v|<u\atop{(u, v)=1}}}
e\left(\frac{m\sqrt{X}}{u^2+2v^2}+\frac{2mv}{u(u^2+2v^2)}-\frac{m\overline{v}_{u}}{u}\right)\nonumber\\
&=\Theta'_m+\Theta''_m\,,
\end{align}
say. Let us consider $\Theta'_m$. Denote
\begin{equation}\label{fu}
f(u)=e\left(\frac{m\sqrt{X}}{u^2+2v^2}-\frac{mu}{v(u^2+2v^2)}\right),
\end{equation}
\begin{equation}\label{eta12}
\eta_1(v)=\sqrt{\max(0,D_2-2v^2)}, \quad \eta_2(v)=\sqrt{\min(v^2,2D_2-2v^2)},
\end{equation}
\begin{equation}\label{Kloostermanvm}
K_{v,m}(t)=\sum\limits_{\eta_1(v)\leq u<t\atop{(u, v)=1}}e\left(\frac{m\overline{u}_{|v|}}{v}\right).
\end{equation}
Using \eqref{Thetamest1} -- \eqref{Kloostermanvm} and Abel's summation formula we obtain
\begin{align}\label{Thetam'est1}
\Theta'_m&=\sum\limits_{\sqrt{\frac{D_2}{3}}\leq |v|<\sqrt{D_2}}
\sum\limits_{\eta_1(v)\leq u<\eta_2(v)\atop{(u, v)=1}}
f(u)e\left(\frac{m\overline{u}_{|v|}}{v}\right)\nonumber\\
&=\sum\limits_{\sqrt{\frac{D_2}{3}}\leq |v|<\sqrt{D_2}}\left( f\big(\eta_2(v)\big)K_{v,m}\big(\eta_2(v)\big)
- \int\limits_{\eta_1(v)}^{\eta_2(v)}K_{v,m}(t)\left(\frac{d}{dt}f(t)\right)\,dt \right)\nonumber\\
&\ll\sum\limits_{\sqrt{\frac{D_2}{3}}\leq |v|<\sqrt{D_2}}\left(1+\frac{m\sqrt{X}}{v^2}\right)
\max_{\eta_1(v)\leq t\leq \eta_2(v)}|K_{v,m}(t)|.
\end{align}
We are now in a good position to apply Lemma \ref{Weilsestimate}
because the sum defined by \eqref{Kloostermanvm} is  incomplete Kloosterman sum.
Thus
\begin{equation}\label{Kloostermanvmest}
K_{v,m}(t)\ll|v|^{\frac{1}{2}+\varepsilon}\,(v,m)^{\frac{1}{2}}\,.
\end{equation}
By \eqref{Thetam'est1} and  \eqref{Kloostermanvmest} we get
\begin{align}\label{Thetam'est2}
\Theta'_m&\ll\sum\limits_{\sqrt{\frac{D_2}{3}}\leq |v|<\sqrt{D_2}}
\left(1+\frac{m\sqrt{X}}{v^2}\right)|v|^{\frac{1}{2}+\varepsilon}\,(v,m)^{\frac{1}{2}}\nonumber\\
&\ll X^\varepsilon\Big(D_2^{\frac{1}{4}}+ mX^{\frac{1}{2}}D_2^{-\frac{3}{4}} \Big)
\sum\limits_{0<v<\sqrt{D_2}}(v,m)^{\frac{1}{2}}.
\end{align}
On the other hand
\begin{equation}\label{sumvm}
\sum\limits_{0<v<\sqrt{D_2}}(v,m)^{\frac{1}{2}}
\leq\sum\limits_{l|m}l^{\frac{1}{2}}\sum\limits_{v\leq \sqrt{D_2}\atop{v\equiv0\,(l)}}1
\ll D_2^{\frac{1}{2}}\sum\limits_{l|m}l^{-\frac{1}{2}}\ll D_2^{\frac{1}{2}}\tau(m)\ll X^\varepsilon D_2^{\frac{1}{2}}.
\end{equation}
The estimations \eqref{Thetam'est2} and \eqref{sumvm} imply
\begin{equation}\label{Thetam'est3}
\Theta'_m\ll X^\varepsilon\Big(D_2^{\frac{3}{4}}+ mX^{\frac{1}{2}}D_2^{-\frac{1}{4}} \Big).
\end{equation}
Proceeding in a similar way for $\Theta''_m$ from \eqref{Thetamest1} we deduce
\begin{equation}\label{Thetam''est}
\Theta''_m\ll X^\varepsilon\Big(D_2^{\frac{3}{4}}+ mX^{\frac{1}{2}}D_2^{-\frac{1}{4}} \Big).
\end{equation}
Now \eqref{Thetamest1}, \eqref{Thetam'est3} and \eqref{Thetam''est} give us
\begin{equation}\label{Thetamest2}
\Theta_m\ll X^\varepsilon\Big(D_2^{\frac{3}{4}}+ mX^{\frac{1}{2}}D_2^{-\frac{1}{4}} \Big).
\end{equation}
From \eqref{Sigma4est2} and \eqref{Thetamest2}  it follows
\begin{equation}\label{Sigma4est3}
\Sigma_4\ll  X^\varepsilon\Big(D_2M_2^{-1}+D_2^{\frac{3}{4}}+X^{\frac{1}{2}}M_2D_2^{-\frac{1}{4}}\Big).
\end{equation}
Taking into account  \eqref{M2}  and  \eqref{Sigma4est3} we find
\begin{equation}\label{Sigma4est4}
\Sigma_4\ll  X^{1+\varepsilon}D_2^{-\frac{1}{4}}.
\end{equation}
Using \eqref{Sigma2est1}, \eqref{Sigma2est2} and \eqref{Sigma4est4} we obtain
\begin{equation}\label{Sigma2est3}
\Sigma_2\ll  X^{1+\varepsilon}D_2^{-\frac{1}{4}}.
\end{equation}

\textbf{Estimation of $\mathbf{\Gamma_2(X)}$}

Summarizing \eqref{DT}, \eqref{Gamma2est2}, \eqref{Gamma2est3}, \eqref{Sigma1est2}
and  \eqref{Sigma2est3}  we get
\begin{equation}\label{Gamma2est4}
\Gamma_2(X)\ll  X^{1+\varepsilon}z^{-\frac{1}{8}}\,.
\end{equation}

\subsection{The end of the proof }
\indent

Bearing in mind \eqref{GammaXdecomp}, \eqref{GammaX1est2}, \eqref{Gamma2est4}
and choosing $z=X^{\frac{8}{9}}$ we establish the asymptotic formula
\eqref{asymptoticformula}.

The theorem is proved.

\vskip20pt
\footnotesize
\begin{flushleft}
S. I. Dimitrov\\
Faculty of Applied Mathematics and Informatics\\
Technical University of Sofia \\
8, St.Kliment Ohridski Blvd. \\
1756 Sofia, BULGARIA\\
e-mail: sdimitrov@tu-sofia.bg\\
\end{flushleft}
\end{document}